\documentclass[12pt]{amsart}

\input xy  \xyoption{all}

\DeclareMathOperator\Fl{\mathrm Fl}
\DeclareMathOperator\PP{\mathcal P}
\DeclareMathOperator\E{\mathcal E}

\DeclareMathOperator\A{\mathbb A}
\DeclareMathOperator\B{\mathbb B}
\DeclareMathOperator\N{\mathbb N}
\DeclareMathOperator\HH{\mathcal H}
\DeclareMathOperator\R{\mathbb R}
\DeclareMathOperator\C{\mathbb C}
\DeclareMathOperator\Z{\mathbb Z}
\DeclareMathOperator\Q{\mathbb Q}

\DeclareMathOperator{\Hom}{Hom}
\DeclareMathOperator{\Ext}{Ext}
\DeclareMathOperator{\codim}{codim}
\DeclareMathOperator{\rk}{rk}
\DeclareMathOperator\CQ{{\mathbb C}Q}
\DeclareMathOperator\EE{\mathbb E}
\DeclareMathOperator\F{\mathcal F}
\DeclareMathOperator\G{\mathcal G}

\newcommand{\ontop}[2]{\genfrac{}{}{0pt}{}{#1}{#2}}
\newcommand{\mapsfrom}{\leftarrow}

\DeclareMathOperator\om{\omega}
\DeclareMathOperator{\DDelta}{\boldsymbol{\Delta}}

\newtheorem{fact}{Fact}[section]
\newtheorem{lemma}[fact]{Lemma}
\newtheorem{theorem}[fact]{Theorem}
\newtheorem{definition}[fact]{Definition}
\newtheorem{example}[fact]{Example}
\newtheorem{rremark}[fact]{Remark}
\newenvironment{remark}{\begin{rremark} \rm}{\end{rremark}}
\newtheorem{proposition}[fact]{Proposition}

\title{On the Cohomological Hall Algebra of Dynkin quivers}

\author{R. Rim\'anyi}
\address{Department of Mathematics, University of North Carolina at Chapel Hill, USA and University of Geneva, Switzerland}
\email{rimanyi@email.unc.edu}

\begin{document}

\begin{abstract} 
Consider the Cohomological Hall Algebra as defined  by Kontsevich and Soibelman, associated with a Dynkin quiver. We reinterpret the geometry behind the multiplication map in the COHA, and give an iterated residue formula for it. We show natural subalgebras whose product is the whole COHA (except in the $E_8$ case). The dimension count version of this statement is an identity for quantum dilogarithm series, first proved by Reineke. We also show that natural structure constants of the COHA are universal polynomials representing degeneracy loci, a.k.a. quiver polynomials.
\end{abstract}

\maketitle

\section{Introduction}

Let $Q=(Q_0,Q_1)$ be an oriented graph, the quiver. For a dimension vector $\gamma\in \N^{Q_0}$ one defines a representation of a Lie group $G_\gamma$  on a vector space $V_\gamma$.
The main object of study of this paper is the formal sum of equivariant cohomology groups
\begin{equation} \label{eqn:oplus}
\bigoplus_{\gamma} H^*_{G_\gamma}(V_\gamma).
\end{equation}

On the one hand, this sum contains geometrically relevant elements, the so-called quiver polynomials. Quiver polynomials are universal polynomials expressing fundamental cohomology classes of degeneracy loci---in the spirit of the celebrated Giambelli-Thom-Porteous formula (cf. also Thom polynomials). They generalize key objects of algebraic combinatorics (different double and quantum versions of Schur and Schubert polynomials)---even if the underlying unoriented graph is one of the Dynkin graphs ADE. Several algorithms and formulas are given in the literature for quiver polynomials (see references in Section \ref{sec:tps}). Yet, quite little is known about their structure; for example a Schur-positivity conjecture of Buch \cite{buch:dynkin} (see also \cite[Sect.8]{RRqr}) has been open for many years now.

On the other hand, in \cite{ks_coha} Kontsevich and Soibelman defined a remarkable associative product on (\ref{eqn:oplus}), and named the resulting algebra the Cohomological Hall Algebra (COHA) of the quiver. Their definition is inspired by string theory, their COHA is a candidate for the algebra of BPS states. In \cite{ks_coha} and in subsequent papers by others, the key point of studying COHAs is their Poincar\'e series. These series have connections with Donaldson-Thomas invariants. An initial phenomenon along these lines is that the comparison of different Poincar\'e series of a Dynkin COHA implies identities among quantum dilogarithm series. These identities have a long history, for most recent results see \cite{reineke_dilog}, \cite{ks_coha}.

\bigskip

The aim of the present paper is to give a detailed study of COHAs associated with a Dynkin quiver. Doing so, we will adopt both points of views: the search for geometrically relevant elements, and counting dimensions (i.e. Poincar\'e series).


After introductions to some background material on spectral sequences, quivers, and COHA, we will address the following five topics:

(1) In Section \ref{sec:dilog} we will prove a special case (in some sense the `extreme' case) of Reineke's quantum dilogarithm identities. We will derive this from a spectral sequence, which is associated to Kazarian in singularity theory. This treatment of the identities is similar to \cite{ks_coha}, but our point of view is somewhat different.

(2) In Section \ref{sec:geo} we give a new geometric interpretation of Kontsevich and Soibelman's product map in the COHA. This argument is based on a construction of Reineke \cite{reineke}. As a personal note let us mention the author's appreciation for the ideas and results in \cite{reineke}: in our view they foreshadow several important results on quiver polynomials and the COHA.

(3) In recent works on formulas of degeneracy loci (\cite{bsz}, \cite{kaza:gysin}, \cite{kaza:noas}, \cite{RRqr}) a new formalism turned out to be useful: replacing equivariant localization formulas with iterated residue formulas. In Section \ref{sec:res} we present the corresponding result for the COHA multiplication.

(4) In Theorem \ref{thm:tp} we will show that quiver polynomials representing degeneracy loci are natural structure constants of COHAs---one of the main results of this paper.

(5) Another main result is a structure theorem for COHAs of Dynkin type $A_n,D_n, E_6$, and $E_7$, Theorem \ref{thm:structure}. Namely, we will show that the COHA can be factored into the product of natural sub-algebras. This result is claimed in \cite{ks_coha} for $A_2$, but is new for other quivers. The key notion of the proof is a well-chosen restriction homomorphism in equivariant cohomology.

\bigskip

\noindent{\bf Acknowledgement.} This work was carried out while the author visited EPF Lausanne and the University of Geneva. I would like to thank these institutions, in particular, T. Hausel and A. Szenes for their hospitality. Special thanks to T. Hausel who introduced me to the topic of COHAs. The author was partially supported by NSF grant DMS-1200685.

\section{The Kazarian spectral sequence---generalities}\label{sec:kss}

Let the Lie group $G$ act on the manifold $V$, and let $V=\eta_1 \cup \eta_2 \cup\ldots$ be a finite stratification by $G$-invariant submanifolds. If we denote $$F_i=\bigcup_{\codim \eta_j \leq i} \eta_j$$
then we have a filtration
$$F_0 \subset F_1 \subset F_2  \subset \ldots \subset F_{\dim V}=V$$
of $V$. Applying the Borel construction ($B_GX=EG\times_G X$), we obtain a filtration
$$B_GF_0 \subset B_GF_1 \subset B_GF_2  \subset \ldots \subset B_GF_{\dim_{\R} V}=B_GV$$
of the Borel construction $B_GV$. We will call the cohomology spectral sequence associated with this latter filtration the {\sl Kazarian spectral sequence} \cite{kazass} of the action, and the stratification, although it is present explicitly or implicitly in various works, e.g. \cite{ab83}. Cohomology is meant with integer coefficients in the whole paper.

We will be concerned with the Kazarian spectral sequence in the particularly nice situation when $G$ is a complex algebraic group acting on the complex vector space $V$, and the $\eta$'s are the obits. In particular, we assume that the representation has finitely many orbits. Stabilizer subgroups of different points in the orbit $\eta$ are conjugate, hence isomorphic. Let $G_\eta$ be the isomorphism type of the stabilizer subgroup of any point in $\eta$.
The properties of the Kazarian spectral sequence we will need in this paper are collected in the following theorem.

\begin{theorem} \label{thm:kaza}
The Kazarian spectral sequence $E_*^{pq}$ converges to $H^*(BG)$, and
$$E_1^{pq}=\mathop{\bigoplus_{\eta \subset V\ \text{orbit}}}_{\codim_{\R} \eta = p} H^q(BG_\eta).$$
\end{theorem}

The convergence claim follows from the fact that $V$ is (equivariantly) contractible. The $E_1$ claim follows from the usual description of the $E_1$ page as relative cohomologies, if one applies excision and the Thom isomorphism to these relative cohomologies. See details e.g. in \cite{kazass}.
\medskip

In general, when one considers such a spectral sequence, there are a few notions deserving special attention.

\begin{itemize}
\item{} A cohomology spectral sequence has a {\sl vertical edge homomorphism} from the limit of the spectral sequence to the 0'th column of the $E_{\infty}$ page. This homomorphism is very relevant when the 0 codimensional stratum is not an orbit. For example, when it is the set of (semi)stable points of the action, this vertical edge homomorphism is called the Kirwan map. A similar `restriction map' will be considered in Section \ref{sec:structure}.
\item{} The differentials of the spectral sequence on the $E_1$ page are $E_1^{pq}\to E_1^{p+1,q}$. Hence the 0'th row of the $E_1$ page is a complex, whose linear generators are in bijection with the strata. In some situations occurring in singularity theory and in knot theory this complex is called the Vassiliev complex. This complex will not play a role in the present paper.
\item{} In favorable situations the spectral sequence degenerates at $E_1$, that is, we have $E_1=E_2=\ldots=E_{\infty}$. This is automatically the case for example, when every odd column and every odd row of $E_1$ vanishes. In this case, by convergence we have $\oplus_{p+q=N} \rk E_1^{pq}=\rk A_N$ where $A_N$ is the $N$'th graded piece of the limit of the spectral sequence. These identities will be considered in Section \ref{sec:dilog}.
\item{} A cohomology spectral sequence has a {\sl horizonal edge homomorphism} from the 0'th row of the $E_{\infty}$ page to the limit of the spectral sequence. Consider the simplest situation when $V$ is contractible (eg. a vector space) and $E_1=\ldots=E_\infty$. Then the 0'th row of $E_\infty$ is a vector space with basis the set of strata. Under the horizontal edge homomorphism the basis vector corresponding to $\eta$ maps to the equivariant fundamental class $[\overline{\eta}]\in H^{\codim_{\R} \eta}(BG)$ of the closure of $\eta$. The horizonal edge homomorphism is called the {\sl Thom polynomial map} in singularity theory, where $[\overline{\eta}]$ is called the Thom polynomial of $\eta$. We will consider the $[\overline{\eta}]$ classes in Section \ref{sec:tps}.
\end{itemize}

\begin{example} \rm
To get familiar with these notions let us discuss the example of $GL_n=GL_n(\C)$ acting on $\C^n$ the usual way: multiplication. There are two orbits, $\eta_0=\{0\}$ (codim $n$) and $\eta_1=\C^n-\{0\}$ (codim 0). We have $G_{\eta_0}=GL_n$ and $G_{\eta_1}\cong GL_{n-1}$ (homotopy equivalence). Hence the $E_1$ page is 0, except we have the cohomologies of $BGL_{n-1}$ in the 0'th column and the cohomologies of $BGL_{n}$ in the $2n$'th column.
Since the odd Betti numbers of $BGL$ spaces are 0, the spectral sequence degenerates at $E_1$. Let us look at the notions itemized above.
\begin{itemize}
\item{} The vertical edge homomorphism $H^*(BGL_n)\to H^*(BGL_{n-1})$ is induced by the inclusion $GL_{n-1}\subset GL_n$, hence it is $\Z[c_1,\ldots, c_n]\to \Z[c_1,\ldots,c_{n-1}]$, $c_n\mapsto 0$.
\item{} The Vassiliev complex is trivial, having only a $\C$ term at positions 0 and $2n$.
 \item{} The identity we obtain for Betti numbers is $b_{2i}(BGL_n)+b_{2(i+n)}(BGL_{n-1})=b_{2(i+n)}(BGL_n)$. Noting that $b_{2i}(BGL_n)$ is the number of partitions of $i$ using only the parts $1,2,\ldots,n$, the identity is combinatorially obvious. A good way to encode this identity for all $i$ at the same time is using the generating sequence $f_n=\sum_i b_{2i}(BGL_n)q^i=1/\prod_{j=1}^n (1-q^j)$:
\begin{equation*}\label{triv_id} f_n+q^n f_{n-1} = q^n f_n.\end{equation*}
\item{} The horizontal edge homomorphism is given by
\begin{equation*}
\begin{array}{lll}
 H^0(BGL_{n-1})\to H^0(BGL_n),  & & 1\mapsto [\overline{\eta_1}]=1, \\
 H^0(BGL_{n})\to H^{2n}(BGL_n), & & 1\mapsto [\eta_0]=c_n.
\end{array}
\end{equation*}
\end{itemize}
\end{example}

\section{Quivers---generalities} \label{sec:quivers}

Let $Q_0$ be the set of vertices, and let $Q_1\subset Q_0\times Q_0$ be the set of edges of a finite oriented graph $Q=(Q_0,Q_1)$, the quiver. Tails and heads of an edge are denoted by $t,h$, that is, $a=(t(a),h(a))\in Q_1$. By dimension vector we mean an element of $\N^{Q_0}$. The Euler form on the set of dimension vectors is defined by
$$\chi(\gamma_1, \gamma_2)= \sum_{i\in Q_0} \gamma_1(i) \gamma_2(i) - \sum_{a\in Q_1} \gamma_1(t(a)) \gamma_2(h(a)).$$
Its opposite anti-symmetrization will be denoted by $\lambda$:
$$\lambda(\gamma_1,\gamma_2)=\chi(\gamma_2,\gamma_1)-\chi(\gamma_1,\gamma_2).$$

\subsection{Path Algebra, Modules} \label{sec:q_modules}
Consider the complex vector space spanned by the oriented paths (including the empty path $\psi_i$ at each vertex $i$) of $Q$. Define the multiplication on this space by concatenation (or 0, if the paths do not match). The resulting algebra is the path algebra $\CQ$ of the quiver $Q$. By $Q$-module we will mean a finite dimensional right $\CQ$-module. 
One has $\Ext^{\geq 2}(M,N)=0$ for any two modules, hence $\Ext^1$ will simply be called $\Ext$. A module $M$ defines a dimension vector $\gamma(i)=\dim (M \psi_i)$. Hence the Euler form can be defined on modules as well, and we have $\chi(M,N)=\dim \Hom(M,N)-\dim \Ext(M,N)$. A module is called indecomposable if it is indecomposable as a $\CQ$-module. Every module $M$ can be written uniquely as $\oplus_\beta c_\beta A_\beta$, where $A_\beta$'s are indecomposable modules.

\subsection{Quiver representations} \label{sec:q_reps}
Fixing a dimension vector $\gamma\in \N^{Q_0}$ we have the quiver representation of the group $G_\gamma=\times_{i\in Q_0} GL_{\gamma(i)}$ on the vector space  $V_\gamma=\Hom_{a\in Q_1}(\C^{\gamma(t(a))},\C^{\gamma(h(a))})$ by
$$ (g_i)_{i\in Q_0} \cdot  (\phi_a)_{a\in Q_1} = (g_{h(a)} \circ \phi_a \circ g_{t(a)}^{-1} )_{a\in Q_1}.$$
The set of orbits of $V_\gamma$ is in bijection with the isomorphism classes of $Q$-modules whose dimension vector is $\gamma$. The complex codimension of the orbit in $V_\gamma$, corresponding to the module $M$, is $\dim \Ext (M,M)$ (Voigt lemma).

\subsection{Dynkin quivers} From now on in the whole paper we assume that the underlying unoriented graph of $Q$ is one of the simply laced Dynkin graphs $A_n, D_n, E_6, E_7, E_8$. These quivers are called Dynkin (or finite) quivers. The simple roots of the same name root system will be denoted by $\alpha_i$ ($i\in Q_0$). The set of positive roots will be denoted by $R(Q)=\{\beta_1,\ldots,\beta_N\}$. Define the non-negative numbers $d_u^i$ by $\beta_u=\sum_{i\in Q_0} d_u^i \alpha_i$. Gabriel's theorem claims that the indecomposable $Q$-modules (up to isomorphism) are in bijection with the positive roots. Moreover, if $A_{\beta_u}$ denotes the indecomposable module corresponding to $\beta_u$, then the dimension vector of $A_{\beta_u}$ is $\gamma(i)=d_u^i$. In particular, for a Dynkin quiver, and any dimension vector $\gamma$, there are only finitely many orbits of the quiver representation $V_\gamma$. For an indecomposable module $A_{\beta_u}$ one has $\chi(A_{\beta_u},A_{\beta_u})=1$.

\subsection{On stabilizers of orbits}
Consider a Dynkin quiver $Q$, and a dimension vector $\gamma\in\N^{Q_0}$. Let $\eta$ be an orbit of the corresponding quiver representation of the group $G_\gamma$ on $V_\gamma$, as defined in Section \ref{sec:q_reps}. Let the corresponding $\CQ$ module be $M$ (Section \ref{sec:q_modules}). Let
$$M=\bigoplus_{\beta\in R(Q)} m_\beta A_\beta$$
be the unique expansion of $M$ into direct sums of indecomposable modules $A_\beta$ and multiplicities $m_\beta\in\N$.

\begin{proposition} \cite[Prop.3.6]{tegez} \label{prop:stab}
Up to homotopy equivalence the stabilizer subgroup $G_\eta$ is
$$G_\eta=\times_{\beta\in R(Q)} U(m_\beta).$$
\end{proposition}

Note that this proposition essentially depends on the fact that the so-called Auslander-Reiten quiver of $Q$ has no oriented cycles, which holds for Dynkin quivers.

\subsection{The quantum algebra of the quiver} Let $q^{1/2}$ be a variable, its square will be denoted by $q$.
The quantum algebra $\A_Q$ of the quiver $Q$ is the $\Q(q^{1/2})$-algebra generated by symbols $y_{\gamma}$ for all $\gamma\in \N^{Q_0}$, and subject to the relations
\begin{equation}\label{eq:AQ_def}
y_{\gamma_1} y_{\gamma_2} = - q^{\frac12\lambda(\gamma_1,\gamma_2)} y_{\gamma_1+\gamma_2}.
\end{equation}
The symbols $y_{\gamma}$ for $\gamma\in \N^{Q_0}$ form a basis of $\A_Q$. As an algebra, the special elements $y_{e_i}$ ($e_i(j)=\delta_{ij}$) generate $\A_Q$.

We will use the shorthand notation $y_M=y_{\gamma(M)}$ (where $\gamma(M)$ is the dimension vector of the module $M$) for a module $M$, as well as $y_\beta=y_{A_\beta}$ for a positive root $\beta$ and the corresponding indecomposable module $A_\beta$.

\section{Convention on the ordering of simple and positive roots} \label{sec:order}

Let $Q$ be a Dynkin quiver with $n$ vertices. The corresponding simple roots are associated to the vertices. We will fix an ordered list of the simple roots
$$\alpha_1,\alpha_2,\ldots,\alpha_n$$
in such a way that the head of any edge comes before its tail. That is, from now on in the whole paper, the vertices of the quiver will be numbered from $1$ to $n$ in such a way that for every edge head comes before tail.

We will also fix an ordering of the positive roots
$$\beta_1,\beta_2,\ldots,\beta_N$$
in such a way that
$$u<v \qquad \Rightarrow \qquad \Hom(A_{\beta_u},A_{\beta_v})=0, \Ext(A_{\beta_v}, A_{\beta_u})=0,$$
where $A_\beta$ is the indecomposable module corresponding to the positive root $\beta$. Such an ordering exits, but is not unique \cite{reineke_feigin, reineke}.

\smallskip

For example, the quiver $\bullet \to \bullet$ will be $2 \to 1$, and the order described above has to be: $\beta_1=\alpha_2$, $\beta_2=\alpha_1+\alpha_2$, $\beta_3=\alpha_1$.

\section{Reading orbit dimensions from the quantum algebra} \label{sec:codim_in_AQ}

Let $\gamma\in \N^{Q_0}$ be a dimension vector. Consider the associated quiver representation of $G_\gamma$ on $V_\gamma$. Let $\eta$ be an orbit, with corresponding module $\oplus_{u=1}^N m_u A_{\beta_u}$. This implies $\gamma(i)=\sum_{u=1}^N m_ud^i_u$. Let the simple and positive roots be ordered according to Section \ref{sec:order}. Define $w$ by the identity

$$y_{\beta_1}^{m_1} y_{\beta_2}^{m_2} \ldots y_{\beta_N}^{m_N} = (-1)^{\sum_u m_u (\sum_i d_u^i-1)} \cdot  q^w \cdot y_{\alpha_1}^{\gamma(1)}  y_{\alpha_2}^{\gamma(2)} \ldots y_{\alpha_n}^{\gamma(n)}\qquad \in \qquad \A_Q.$$

\begin{lemma} \label{lem:codim_in_AQ}
We have
$$\frac{\sum_{u=1}^N m_u^2}{2} - \frac{\sum_{i=1}^n \gamma(i)^2}{2} + w - \codim_{\C} \eta =0.$$
\end{lemma}

\begin{proof}
First we make calculations to get an expression for $w$. Formula (\ref{eq:AQ_def}) implies
\begin{equation} \label{eqn:AQ_comm}
y_{\gamma_1} y_{\gamma_2} = q^{\lambda(\gamma_1,\gamma_2)} y_{\gamma_2} y_{\gamma_1}
\end{equation}
Let $\beta=\sum d^i\alpha_i$ be a positive root. By repeated applications of (\ref{eq:AQ_def}) and (\ref{eqn:AQ_comm}) we obtain
$$y_\beta^m=(-1)^{m(\sum_i d^i -1)} \cdot  q^{-\frac{m^2}{2} \sum_{i<j\leq n} d^i d^j \lambda(\alpha_i,\alpha_j)} \cdot y_{\alpha_1}^{md^1} \cdots y_{\alpha_n}^{md^n}.$$
Using this for $\beta=\beta_1,\ldots,\beta_N$, and applying (\ref{eqn:AQ_comm}) further we get
$$y_{\beta_1}^{m_1} y_{\beta_2}^{m_2} \ldots y_{\beta_N}^{m_N} =
(-1)^{\sum_u m_u (\sum_i d_u^i-1)} \cdot q^w \cdot y_{\alpha_1}^{\gamma(1)} \ldots y_{\alpha_n}^{\gamma(n)} ,$$
where
$$w= -\sum_{u=1}^N \frac{m_u^2}{2} \sum_{i<j\leq n} d_u^i d_u^j \lambda(\alpha_i,\alpha_j) -
\sum_{u<v\leq N} m_u m_v \sum_{i<j\leq n} d_v^i d_u^j \lambda(\alpha_i,\alpha_j).$$
Second, we need an expression for $\codim \eta$. Using the property we required for the order of $\beta$'s we have
$$\codim_{\C} \eta = \dim \Ext(m_u A_{\beta_u}, m_uA_{\beta_u})=\sum_{u<v}m_um_v\dim \Ext(A_{\beta_u}, A_{\beta_v})= - \sum_{u<v}m_um_v \chi(A_{\beta_u},A_{\beta_v}).$$
Therefore, the coefficient of $m_u m_v$ $(u<v)$ in the expression in the Lemma is
\begin{equation}\label{eqn:a1}
0-\sum_{i=1}^n d_u^i d_v^i -\sum_{i<j} d_v^i d_u^j \lambda(\alpha_i,\alpha_j)  +  \chi(A_{\beta_u}, A_{\beta_v}).
\end{equation}
Notice that for $i<j$ the number $\lambda(\alpha_i,\alpha_j)$ is minus the number of arrows from $j$ to $i$, and there are no arrows from $i$ to $j$ (c.f. the ordering of the $\alpha$'s in Section \ref{sec:order}). Hence the sum of the first three terms of (\ref{eqn:a1}) equals $-\chi(A_{\beta_u}, A_{\beta_v})$ by definition. Hence expression (\ref{eqn:a1}) is 0.
The coefficient of $m_u^2$ in the expression in the Lemma is
\begin{equation}\label{eqn:semmi}
\frac{1}{2} \left( 1 - \sum_{i=1}^n (d^i_u)^2 - \sum_{i<j} d_u^i d_u^j \lambda(\alpha_i,\alpha_j) -0 \right)=\frac12\left( 1 - \chi(A_{\beta_u},A_{\beta_u})\right)=0.
\end{equation}
Here, again, we used that if $i<j$ then $\lambda(\alpha_i, \alpha_j)$ is minus the number of arrows from $j$ to $i$ and there are no arrows from $i$ to $j$. \end{proof}

\section{Quantum Dilog identities from the Kazarian spectral sequence}\label{sec:dilog}

Consider the quantum dilogarithm series
$$
\EE(z) = \sum_{n=0}^{\infty}  \frac{(-1)^n z^n \cdot q^{{n^2/2}}}{(1-q)(1-q^2)\cdots (1-q^n)}.
$$
A remarkable infinite product expression (not used in the present paper) is
$$\EE(z)=(1-q^{1/2}z)(1-q^{3/2}z)(1-q^{5/2}z)(1-q^{7/2}z)\ldots.$$
Putting $f_n=\sum_{i=0}^\infty q^i \dim H^{2i}(BGL_n)$, we can rewrite
$$
\EE(z)= \sum_{n=0}^\infty (-1)^n z^n q^{n^2/2} f_n.
$$

For the history and rich properties of quantum dilogarithm series see e.g. \cite{zagier, keller} and references their. In this section we will reprove a special case of Reineke's $\EE$-identities \cite{reineke_dilog} see also \cite{keller}. These identities generalize some earlier famous results, such as \cite{sch,fad1,fad2}.

\begin{theorem} (Reineke) \label{thm:dilog}
For a Dynkin quiver order the simple and positive roots satisfying the conditions in Section \ref{sec:order}. In the quantum algebra of the quiver we have the identity
\begin{equation}
\EE(y_{\alpha_1}) \EE(y_{\alpha_2})  \cdots \EE(y_{\alpha_n})  = \EE(y_{\beta_1}) \EE(y_{\beta_2}) \cdots \EE(y_{\beta_N}).
\end{equation}
\end{theorem}

\begin{proof} Let $\gamma(1),\ldots,\gamma(n)$ be non-negative integers. We will consider the coefficient of
$y_{\alpha_1}^{\gamma(1)}\cdots y_{\alpha_n}^{\gamma(n)}$ on the two sides. On the left hand side this coefficient is obviously
\begin{equation}\label{eqn:LHS}
(-1)^{ \sum_i \gamma(i)} q^{ \frac12 \gamma(i)^2} \cdot f_{\gamma(1)} \cdots f_{\gamma(n)}.
\end{equation}
On the right hand side we need to write a monomial in $y_{\beta_u}$'s as a monomial in the $y_{\alpha_i}$'s. This is solved in Lemma \ref{lem:codim_in_AQ}. Hence, for the coefficient on the right hand side we obtain
\begin{equation} \label{eqn:RHS}
\sum_m (-1)^{\sum m_u} q^{\frac12 \sum_u m_u^2} f_{m_1} \cdots f_{m_N} \left( (-1)^{\sum_i m_i(\sum_u d^i_u -1)}
q^{- \frac12 \sum_u m_u^2} q^{\frac12 \sum_i \gamma(i)^2} q^{\codim_{\C} \eta_m}\right),
\end{equation}
where the summation runs for those $m=(m_1,\ldots,m_N)$ for which $\sum_{u=1}^N m_u d_u^i=\gamma(i)$ for all $i$. As before, $\eta_m$ is the orbit in the quiver representation corresponding to the module $\sum_u m_u A_{\beta_u}$.
The expression (\ref{eqn:RHS}) is further equal to
\begin{equation}\label{eqn:RHS2}
\sum_m (-1)^{\sum_i \gamma(i)} q^{\frac12 \sum_i \gamma(i)^2} q^{\codim_{\C} \eta_m} \cdot f_{m_1} \cdots f_{m_N}.
\end{equation}
We need to show that expressions (\ref{eqn:LHS}) and (\ref{eqn:RHS2}) are equal.

Consider the Kazarian spectral sequence for the representation with dimension vector $\gamma$. The orbits of this representation have contributions to the $E_1$ page. Namely, let $m=(m_1,\ldots,m_N)$ be such that $\sum_u m_u d^i_u=\gamma(i)$ for all $i$. Then $\eta_m$ is an orbit, and its contribution to the $E_1$ page is (see Theorem \ref{thm:kaza})
$$E_1^{\codim_{\R}\eta_m , j} = H^j(BG_{\eta_m}).$$
Recall from Proposition \ref{prop:stab} that $G_{\eta_m}=\times_u GL_{m_u}$.
Thus, the spectral sequence degenerates at $E_1$. The limit of the spectral sequence is
$H^*(B ( GL_{\gamma(1)} \times \ldots \times GL_{\gamma(n)}))$.

Therefore we obtain an identity for the Betti numbers:
\begin{equation} \label{eqn:fromKss}
\sum_m q^{\codim_{\C} \eta_m} f_{m_1} \cdots f_{m_N} = f_{\gamma(1)} \cdots f_{\gamma(n)}.
\end{equation}
Identity (\ref{eqn:fromKss}) shows that (\ref{eqn:LHS}) is indeed equal to (\ref{eqn:RHS2}).
\end{proof}

\section{COHA of $Q$}

In this section we follow \cite{ks_coha}, and repeat the definition of the Cohomological Hall Algebra (without potential) associated with $Q$.

For a dimension vector $\gamma\in \N^{Q_0}$ define $\HH_{\gamma}=H^*_{G_\gamma}(V_\gamma)$. As a vector space, the COHA of $Q$ is $$\HH=\bigoplus_{\gamma\in \N^{Q_0}} \HH_{\gamma}.$$

\subsection{Geometric definition of the multiplication \cite{ks_coha}}

Let $\gamma_1$ and $\gamma_2$ be dimension vectors. Let the group $G_{\gamma_1,\gamma_2}$ be the subgroup of $G_{\gamma_1+\gamma_2}$ containing $n$-tuples of matrices such that the matrix at vertex $i$ keeps $\C^{\gamma_1}\subset \C^{\gamma_1+\gamma_2}$ invariant (that is, it is upper block-diagonal of size $\gamma_1$, $\gamma_2$). Let $V_{\gamma_1,\gamma_2}$ be the subspace of $V_{\gamma_1+\gamma_2}$ containing linear maps such that the map at edge $a$ maps $\C^{\gamma_1(t(a))}$ into $\C^{\gamma_1(h(a))}$.

The multiplication $*:\HH_{\gamma_1} \otimes \HH_{\gamma_2} \to \HH_{\gamma_1+\gamma_2}$ is defined as the composition
$$H^*_{G_{\gamma_1}}(V_{\gamma_1}) \otimes H^*_{G_{\gamma_2}}(V_{\gamma_2}) \xrightarrow{\times}
H^*_{G_{\gamma_1}\times G_{\gamma_2}}(V_{\gamma_1}\oplus V_{\gamma_2})\xrightarrow{\cong}\hskip 5 true cm \ $$
$$\ \hskip 5 true cm H^*_{G_{\gamma_1,\gamma_2}}(V_{\gamma_1,\gamma_2}) \xrightarrow{\iota_*}
H^*_{G_{\gamma_1,\gamma_2}}(V_{\gamma_1+\gamma_2}) \xrightarrow{\pi_*}
H^*_{G_{\gamma_1 +\gamma_2}}(V_{\gamma_1+\gamma_2}).$$
Here the first map is the product map of algebraic topology. The second map is induced by the obvious equivariant homotopy equivalence. The third map is the push-forward with respect to the embedding $\iota: V_{\gamma_1,\gamma_2} \to V_{\gamma_1+\gamma_2}$. The last map is the push-forward with respect to the fibration $\pi: BG_{\gamma_1,\gamma_2} \to BG_{\gamma_1+\gamma_2}$ (with fiber $G_{\gamma_1+\gamma_2}/G_{\gamma_1,\gamma_2}$).

The multiplication $*$ induced on $\HH$ is associative. Although it respects the dimension vector grading, its relation with the cohomological degree grading is
$$H^{k_1}_{G_{\gamma_1}}(V_{\gamma_1}) * H^{k_2}_{G_{\gamma_2}}(V_{\gamma_2 })  \subset H^{k_1+k_2-2\chi(\gamma_1,\gamma_2)}_{G_{\gamma_1+\gamma_2}}(V_{\gamma_1+\gamma_2}).$$

\subsection{Equivariant Localization formula for multiplication \cite{ks_coha}} \label{sec:localization}

Consider
\begin{equation}
\begin{split}
\HH_\gamma=& \C[\om_{1,1},\ldots, \om_{1,\gamma(1)}, \ \ \ldots\ \ ,\om_{n,1},\ldots, \om_{n,\gamma(n)} ]^{S_{\gamma(1)}\times \ldots \times S_{\gamma(n)}}\\
\end{split}
\end{equation}
where $\om_{i,j}$ are the ``universal'' Chern roots of the tautological bundles over $BGL_{\gamma(i)}$'s. 

Let $f_1\in \HH_{\gamma_1}$ and $f_2\in \HH_{\gamma_2}$. We have

\begin{equation} \label{eqn:localization}
f_1 * f_2 =
\sum_{S_1 \in \binom{ [\gamma_1(1)+\gamma_2(1)] }{\gamma_1(1)} }\ldots \sum_{S_n \in \binom{ [\gamma_1(n)+\gamma_2(n)] }{\gamma_1(n)} }
f_1(\omega_{i,S_i}) f_2(\omega_{i,\bar{S}_i})
\frac{   \prod_{a\in Q_1}  \left( {\omega}_{h(a),\bar{S}_{h(a)}} - \omega_{t(a),S_{t(a)}} \right) }
{\prod_{i\in Q_0}  \left( {\omega}_{i,\bar{S}_i} - \omega_{i,S_i} \right) },
\end{equation}
where $S_i\in \binom{[\gamma_1(i)+\gamma_2(i)]}{\gamma_1(i)}$ means that $S_i$ is a $\gamma_1(i)$-element subset of $\{1,\ldots,\gamma_1(i)+\gamma_2(i)\}$; $\bar{S_i}$ means the complement set $\{1,\ldots,\gamma_1(i)+\gamma_2(i)\}-S_i$. By $(\omega_{j,\bar{S}_j} - \omega_{i,S_i})$ we mean the
product $\prod_{u\in \bar{S}_j} \prod_{v\in S_i} (\omega_{j,u} - \omega_{i,v})$.

This localization formula has an obvious, but notation heavy, generalization for multi-factor products $f_1 * \ldots * f_r$.

\medskip

For example, for the quiver $1 \leftarrow 2$ one has
$$ \left( 1 \in \HH_{10} \right) * \left( 1 \in \HH_{01} \right) =  1 \hskip 1.6 true cm \in \HH_{11}=\C[\omega_{1,1},\omega_{2,1}] ; $$
$$ \left( 1 \in \HH_{01} \right) * \left( 1 \in \HH_{10} \right) =  \omega_{1,1}-\omega_{2,1} \in \HH_{11} =\C[\omega_{1,1},\omega_{2,1}]. $$

\section{Another geometric interpretation of the COHA multiplication} \label{sec:geo}

Let $\gamma_1, \ldots, \gamma_r$ be dimension vectors, and let $\gamma=\sum_{u=1}^r \gamma_u$. Let $f_u\in \HH_{\gamma_u}$. In this section we give another geometric interpretation of the COHA product $f_1 * \ldots * f_r \in \HH_\gamma$.

\subsection{First version}

Given nonnegative integers $\lambda_1,\ldots,\lambda_r$ let $\Fl_{\lambda}$ be the flag manifold parameterizing flags
$$ \{0\}=V_0 \subset V_1 \subset V_2 \subset \ldots \subset V_r=\C^{\sum \lambda_u}$$
with $\dim V_u/V_{u-1} = \lambda_u$. Set
$$\Fl=\Fl_{\gamma_1,\ldots,\gamma_r}=\times_{i=1}^n \Fl_{\gamma_1(i),\ldots,\gamma_r(i)}.$$
The tautological rank $\sum_{v=1}^u \gamma_v(i)$ bundle over $\Fl_{\gamma_1(i),\ldots,\gamma_r(i)}$ will be denoted by $\E_{i,u}$ and we set $\F_{i,u}=\E_{i,u}/\E_{i,u-1}$. We have $\rk \F_{i,u}=\gamma_u(i)$. The $\E_{..}$ and $\F_{..}$  bundles pulled back over $\Fl$ will be denoted by the same letter. Denote
$$\G= \bigoplus_{a\in Q_1} \bigoplus_{u<v} \Hom( \F_{t(a),u}, \F_{h(a),v}).$$
Observe that the group $G_{\gamma}$ acts on $\Fl$. The equivariant Euler class of $\G$ will be denoted by $e(\G)\in H_{G_\gamma}^*(\Fl)$.

\begin{lemma} \label{lem:star_int}
For $f_u\in \HH_{\gamma_u}$,  $u=1,\ldots,r$, we have
\begin{equation}\label{eq:star_with_int}
f_1 * \ldots * f_r = \int_{\Fl} \ \ \prod_{u=1}^r f_u(\F_{.,u}) \cdot e(\G).
\end{equation}
\end{lemma}
Since $f_u\in \HH_{\gamma_u}$, we can evaluate $f_u$ on a sequence of bundles of ranks $\gamma_u(1),$ $\ldots,$ $\gamma_u(n)$. Hence $f_u(\F_{.,u})\in H_{G_\gamma}^*(\Fl)$ makes sense. The integral $\int_{\Fl}$ is the standard push-forward map $H^*_{G_\gamma}(\Fl) \to H^*_{G_\gamma}($point$)$ in equivariant cohomology. The lemma holds because the localization formula for the map in the Lemma is the same as (the multi-factor version of) (\ref{eqn:localization}).

\subsection{An improved version} \label{sec:improved}
Consider the projection $\pi$ to the second factor
$$\pi: \Fl_{\gamma_1,\ldots,\gamma_r} \times V_\gamma \to V_\gamma.$$

\begin{lemma} \label{lem:star_int_impr}
For $f_u\in \HH_{\gamma_u}$, $u=1,\ldots,r$, we have
\begin{equation}\label{eq:star_with_int2}
f_1 * \ldots * f_r = \pi_*\left( \ \ \prod_{u=1}^r f_u(\F_{.,u}) \cdot e(\G) \right).
\end{equation}
\end{lemma}
\noindent This lemma is equivalent to Lemma \ref{lem:star_int}, since $V_\gamma$ is (equivariantly) contractible.
An advantage of this version is the existence of a natural subvariety in the total space whose equivariant fundamental class is  $e(\G)$. Namely, define the ``consistency subset'' (c.f. \cite{reineke})
$$\Sigma =
\left\{
\left(
(V_{i,u})_{{i=1,\ldots,n}\atop {u=1,\ldots,r}},
(\phi_a)_{a\in Q_1} \right)
\in \Fl_{\gamma_1,\ldots,\gamma_r} \times V_\gamma
:
\phi_a(V_{t(a),u})\subset V_{h(a),u}
\
\forall a,u
\right\}
$$
of $\Fl_{\gamma_1,\ldots,\gamma_r} \times V_\gamma$.

\begin{lemma} \label{lem:Sigma-lemma}
We have $[\Sigma]=e(\G)$.
\end{lemma}

\begin{proof}
Let $p$ be a torus fixed point of $\Sigma$. In a neighborhood of $p$ let us choose subbundles $\bar{\F}_{i,u} \subset \E_{i,u}$ such that $\bar{\F}_{i,u}\oplus \E_{i,u-1}=\E_{i,u}$. We have a tautological section $\Theta$ of the bundle $\bar{\G}=\oplus_{a\in Q_1} \oplus_{u<v} \Hom(\bar{\F}_{t(a),u},\F_{h(a),v})$. The zero-section of $\Theta$ is exactly $\Sigma$, and it can be shown that $\Theta$ is transversal to the 0-section. Hence, at $p$ we have $[\Sigma]|_p=e(\bar{\G})|_p=e(\G)|_p$. Since this holds at every torus fixed point, we have $[\Sigma]=e(\G)$.
\end{proof}

\section{Another formula for the COHA multiplication} \label{sec:res}

In recent works on Thom polynomials of singularities as well as on formulas for quiver polynomials certain iterated residue descriptions turned out to be useful. This method was pioneered in \cite{bsz}, then worked out in \cite{kaza:gysin}, see also \cite[Sect.11]{ts_loc}, \cite{lengyel}, \cite{RRqr}.

In this section we show the iterated residue formula for the COHA multiplication. Since we will not need this formula in the rest of the paper, we will not give a formal proof how to turn a localization formula to a residue formula.

\begin{definition}
Let $Q$ be a Dynkin quiver and $\gamma$ a dimension vector.
For a vertex $i\in Q_0$ define its tail $T(i)=\{j\in Q_0: \exists (j,i)\in Q_1\}$.
For $\lambda\in\Z^r$ define
$$\Delta_\lambda^{(i)}=\det\left( c_{i,\lambda_u+v-u} \right)_{u,v=1,\ldots,r} \in \HH_{\gamma},$$
where $c_{i,<0}=0$ and
$$c_{i,0}+c_{i,1}\xi + c_{i,2} \xi^2 + \ldots =
\frac{ \prod_{j\in T(i)}\prod_{u=1}^{\gamma(j)} (1-\omega_{j,u}\xi)}
{\prod_{u=1}^{\gamma(i)} (1-\omega_{i,u}\xi)}.$$
\end{definition}

Let $\A_i=(a_{i,1},a_{i,2},\ldots, a_{i,r_i} )$ be ordered sets of variables for $i=1,\ldots,n$. The following operation can be called ``Jacobi-Trudi transform'' \cite{ryan} or iterated residue operation \cite{RRqr}.

\begin{definition}
For a Laurent monomial in the variables $\cup_i \A_i$ define
\[
\DDelta
\left(\prod_{i=1}^n \prod_{s=1}^{r_j} a_{i,s}^{\lambda_{i,s}}
 \right)=
\DDelta_{\A_1,\ldots,\A_p}
\left(
\prod_{i=1}^n \prod_{s=1}^{r_i} a_{i,s}^{\lambda_{i,s}}
 \right)
 =
\prod_{i=1}^n \Delta_{\lambda_{i,1},\ldots,\lambda_{i,r_i}}^{(i)}
\]
For an element of $\Z[[a_{ks}^{\pm1}]]$, which has finitely many monomials with non-0 $\DDelta$-value, extend this operation linearly.
\end{definition}

\medskip

Let $f_1\in \HH_{\gamma_1}$, $f_2\in \HH_{\gamma_2}$. Let $\A_i$ and $\B_i$ be sets of variables with $|\A_i|=\gamma_1(i)$, $|\B_i|=\gamma_2(i)$.
Suppose $f_1=\Delta_{\A_1,\ldots,\A_N}(g)$ for some function $g(a_{.,.})$. Let $k_i=\sum_{j\in T(i)} |\A_j| - |\A_i|$.

\begin{theorem}
We have
\begin{equation}
f_1 * f_2 = \DDelta_{\B_1\A_1,\ldots,\B_n\A_n}\left(
g \cdot \frac{ \prod_i \B_i^{k_i} \cdot f_2(\B)}{\prod_{a\in Q_1} \left( 1- \frac{\B_{t(a)}}{\A_{h(a)}} \right)\left( 1- \frac{\B_{t(a)}}{\B_{h(a)}} \right)}
\right) \in \HH_{\gamma_1+\gamma_2}.
\end{equation}
\end{theorem}

Here we used obvious multiindex notations, such as
$$\B^k=\prod_{b\in \B} b^k, \qquad\qquad\qquad \left(1-\frac{\B}{\A}\right)=\prod_{b\in \B}\prod_{a\in \A} \left(1-\frac{b}{a}\right).$$

\begin{example} Let $Q=1 \leftarrow 2$, $f_1=1\in \HH_{01}, f_2=1\in \HH_{10}$. Let $\A_1=\{\}, \A_2=\{a_{21}\}$, $\B_1=\{b_{11}\}, \B_2=\{\}$. We have $f_1=\DDelta_{\A_1,\A_2}(1)$, hence
$$f_1*f_2=\DDelta_{   \{b_{11}\}, \{a_{21}\} }
\left(
1 \cdot \frac{ b_{11}^1 \cdot 1}{1}
\right)=
\DDelta_{1}^{(1)}=
\frac{1-\omega_{2,1}\xi}{1-\omega_{1,1}\xi}|_{1}=\omega_{1,1}-\omega_{2,1}.$$
\end{example}

\section{Fundamental classes of orbit closures in the COHA}\label{sec:tps}

Recall that $Q$ is a Dynkin quiver and $\alpha_i$ and $\beta_u$ are the simple, resp. positive roots of the same named root system, listed in the order specified in Section \ref{sec:order}. Let $M_m=\sum_u m_u A_{\beta_u}$ be a $\C Q$-module, and let $\eta_m$ be the corresponding orbit in a quiver representation $G_\gamma$ acting on $V_\gamma$. In particular $\sum m_u d^i_u=\gamma(i)$ for all $i$.

A remarkable object associated with the orbit $\eta_m$ is the  equivariant fundamental  class $[\overline{\eta}_m]\in \HH_{\gamma}$ of its closure. The rich algebraic combinatorics of this class---called a quiver polynomial--- is studied e.g. in \cite{buch-fulton, tegez, B:Gr,   bkty,bfr, BSY, ks, KMS, br, buch:dynkin, RRqr}. Now we show that this class is a natural structure constant of the COHA.

\begin{theorem} \label{thm:tp}
The fundamental $G_\gamma$-equivariant class of the orbit closure $\overline{\eta}_m$ in $\HH_{\gamma}=H^*_{G_\gamma}(V_\gamma)$ is
$$ [\overline{\eta}_m]=  \left( 1\in \HH_{m_1\beta_1} \right) * \left( 1\in \HH_{m_{2}\beta_{2}} \right) * \ldots * \left( 1\in \HH_{m_N\beta_N} \right)
\in \HH_{\gamma}.$$
\end{theorem}

\begin{proof}
Consider the construction of Section \ref{sec:improved} for the dimension vectors $m_1\beta_1, \ldots, m_N\beta_N$ (in this order). Reineke proved in \cite{reineke} that the projection $\pi|_\Sigma:\Sigma \to \overline{\eta}_m$ is a resolution of the orbit closure $\overline{\eta}_m$. Hence we have
$[\overline{\eta}_m]=\pi_*([\Sigma])$. On the other hand $\pi_*([\Sigma])=( 1\in \HH_{m_1\beta_1} ) *  \ldots * ( 1\in \HH_{m_N\beta_N} )$ because of Lemmas \ref{lem:star_int_impr} and \ref{lem:Sigma-lemma}.
\end{proof}

\section{Structure of Dynkin COHAs} \label{sec:structure}

Let $Q$ be a Dynkin quiver, but {\em not} an orientation of the graph $E_8$. Let $\alpha_i$ and $\beta_u$ be the simple, resp. positive roots of the same named root system, listed in the order specified in Section \ref{sec:order}. Recall that $\beta_u=\sum_i d^i_u \alpha_i$.

For each $u$ choose an $i$ such that $d^i_u=1$, and call this $i=i(u)$. This choice will be fixed throughout the section, and will not be indicated in notation. (This argument does not work for $E_8$: the longest positive root of $E_8$ does not admit such an $i$.)

\begin{definition} Let
$$\overline{\PP}_{\beta_u}=\{ f(\omega_{i(u),1} ) \} \subset \HH_{\beta_u}.$$
Let  $\PP_{\beta_u}$ be the subring of $\HH$ generated by $\overline{\PP}_{\beta_u}$.
\end{definition}

In other words, $\PP_{\beta_u}$ consists of equivariant classes in $\HH_{m\beta_u}$ for all $m=0,1,2,\ldots$ that only depend on the Chern roots (or Chern classes) at the $i(u)$'th vertex. For a simple root $\alpha_i$ the subring $\PP_{\alpha_i}$ is generated by $\overline{\PP}_{\alpha_i}=\HH_{\alpha_i}$. We clearly have
$$\PP_{\beta_u} \cong \HH^{A_1}.$$

\begin{theorem} \label{thm:structure} Let $Q$ be a Dynkin quiver, but not an orientation on $E_8$, and use the notations above. In particular $\PP_\beta\cong \HH^{A_1}$ are subrings of $\HH^Q$.
The $*$ multiplication (from left to right) induces {\em isomorphisms}
\begin{equation} \label{eqn:str1}  \PP_{\alpha_1} \otimes \PP_{\alpha_{2}} \otimes \ldots \otimes \PP_{\alpha_n} \xrightarrow{*} \HH^Q,\end{equation}
\begin{equation} \label{eqn:str2}  \PP_{\beta_1} \otimes \PP_{\beta_{2}} \otimes \ldots \otimes \PP_{\beta_N} \xrightarrow{*} \HH^Q. \end{equation}
\end{theorem}

\begin{remark} The (\ref{eqn:str1}) part of the Theorem holds for an $E_8$ quiver as well. Let $\rho$ be the longest root of $E_8$. Let us fix an $i\in \{1,\ldots,8\}$ and define the subspace $\PP^{E_8}_{i,\rho}=\{ f\in \HH_{m\rho} : m=0,1,2,\ldots, f=f(\alpha_{i,1})\}$. The analogue of (\ref{eqn:str2}) holds for an $E_8$ quiver as well, if we use $\PP^{E_8}_{i,\rho}$ for $\PP_{\rho}$. However, the subspace $\PP^{E_8}_{i,\rho}$ is not a subring of $\HH^{E_8}$.
\end{remark}

Theorem \ref{thm:structure} for $Q=A_2$ was announced in \cite{ks_coha}. The rest of Section \ref{sec:structure} is devoted to proving Theorem \ref{thm:structure}.

\medskip

First, let $f_i(\omega_{i,1},\ldots, \omega_{i,\gamma(i)}) \in \PP_{\alpha_i}$. From the localization formula (\ref{eqn:localization}) (or simply from the topological definition of multiplication), it follows that
$$f_1 * \ldots * f_n = f_1 \cdot \ldots \cdot f_n \in \HH_{\gamma}=
\Z[\omega_{i,j}]_{i=1,\ldots,n \atop j=1,\ldots,\gamma(i)}^{S_{\gamma(1)}\times \ldots \times S_{\gamma(n)}}.$$
This implies that (\ref{eqn:str1}) is an isomorphism.

\subsection{Injectivity} \label{sec:injectivity}

Next we want to show that the map in (\ref{eqn:str2}) is injective. Let $m_1,\ldots,m_N$ be nonnegative integers. Let the dimension vector of $\sum_u m_u \beta_u$ be $\gamma$. At the action of $G_\gamma$ acting on $V_\gamma$, let the orbit corresponding to $\sum_u m_u A_{\beta_u}$ be $\eta_m$. Set
$\PP_{\beta,m}=\PP_{\beta} \cap \HH_{m\beta}$.

\begin{lemma} \label{lem:inj1}
The map induced by $*$ multiplication
$$\phi_m: \PP_{\beta_1,m_1} \otimes \ldots \otimes \PP_{\beta_N,m_N} \to \HH_{\gamma}$$
is injective.
\end{lemma}

\begin{proof}
Let $Y_{i,u,v}$  for $i=1,\ldots,n$, $u=1,\ldots,N$, $v=1,\ldots,m_u$ be sets of non-negative integers with $|Y_{i,u,v}|=d^i_u$, and such that the {\em disjoint} union
$$ Y_{i,1,1} \cup \ldots \cup Y_{i,1,m_1} \ \ \cup\ \
Y_{i,2,1} \cup \ldots \cup Y_{i,2,m_2} \ \ \cup \ \
\ldots
\ \ \cup\ \
Y_{i,N,1} \cup \ldots \cup Y_{i,N,m_N}
$$
is equal to $\{1,\ldots,\gamma(i)\}$ {\em in this order}. That is,
$$Y_{i,1,1}=\{1,\ldots,d^i_1\},\ \  Y_{i,1,2}=\{d^i_1+1,\ldots,2d^i_1\},\ \ \text{etc}.$$

Next we will associate an element $\Phi\in V_\gamma$ to the system of sets $Y_{i,u,v}$. Let $e_{i,1},\ldots,e_{i,\gamma(i)}$ be the standard basis of $\C_i^{\gamma(i)}$. Let $A_{u,v}$ be an indecomposable $\C Q$ module isomorphic with $A_{\beta_u}$ spanned by $e_{i,j}$ for $j\in Y_{i,u,v}$. Set $\Phi=\oplus_{u=1}^N \oplus_{v=1}^{m_u} A_{u,v}$.

\medskip
\noindent{\sl Example:} Let $Q=1 \leftarrow 2$, $\beta_1=\alpha_2, \beta_2=\alpha_1+\alpha_2, \beta_3=\alpha_1$ and $m_1=m_2=m_3=2$. The reader might find the following diagram---illustrating the role of the indexing sets---helpful.
\begin{equation*}
\begin{array}{ccccccc}
Y_{1,1,1}=\{\} &  &         &             & e_{2,1} &       & Y_{2,1,1}=\{1\} \\
Y_{1,1,2}=\{\} &  &         &             & e_{2,2} &       & Y_{2,1,2}=\{2\} \\
Y_{1,2,1}=\{1\} & & e_{1,1} &  \mapsfrom  & e_{2,3} & & Y_{2,2,1}=\{3\} \\
Y_{1,2,2}=\{2\} & & e_{1,2} &  \mapsfrom  & e_{2,4} & & Y_{2,2,2}=\{4\} \\
Y_{1,3,1}=\{3\} & & e_{1,3} &             &  & & Y_{2,3,1}=\{\} \\
Y_{1,3,2}=\{4\} & & e_{1,4} &             &  & & Y_{2,3,2}=\{\} \\
\end{array}
\end{equation*}
Hence $\Phi:\C^4_2 \to \C^4_1$ is a rank 2 map with kernel span$(e_{2,1},e_{2,1})$ and $e_{2,3}\mapsto e_{1,1}$, $e_{2,4}\mapsto e_{1,2}$.
\medskip

The restriction map $\iota_m^*:H_{G_\gamma}^*(V_\gamma)\to H_{G_\gamma}^*(\eta_m)$ induced by $\iota_m: \eta_m \subset V_\gamma$ can be identified with
$H^*(BG_\gamma) \to H^*(BG_{\Phi})$ where $G_{\Phi}$ is the stabilizer subgroup of the $\Phi$ in $G_\gamma$. This map is
\begin{equation} \label{eqn:restriction}
\begin{array}{ccccc}
 \iota_m^*: \C[\omega_{i,j}]^{S_\gamma(1)\times \ldots\times S_\gamma(n)}_{  \ontop{i=1,\ldots,n}{j=1,\ldots,\gamma(i)} } & \to &
 \C[\mu_{u,v}]^{S_{m_1}\times \ldots\times S_{m_N}}_{  \ontop{u=1,\ldots,N}{v=1,\ldots,m_u} } & &  \\
 \omega_{i,j} & \mapsto & \mu_{u,v} & \text{if} & j\in Y_{i,u,v}.
\end{array}
\end{equation}
The restriction map $\iota_m^*$ is studied in detail in \cite[Sect.3]{tegez}.

\medskip
\noindent{\sl Example:} In the above example we have
\begin{equation}\label{eqn:iota}
\begin{array}{cccccc}
 &  & \mu_{1,1} & \mapsfrom& \omega_{2,1}\\
 &  & \mu_{1,2} & \mapsfrom & \omega_{2,2}\\
\omega_{1,1} & \mapsto & \mu_{2,1} & \mapsfrom &  \omega_{2,3}& \\
\omega_{1,2} & \mapsto & \mu_{2,2} & \mapsfrom &  \omega_{2,4}& \\
\omega_{1,3}             & \mapsto & \mu_{3,1} & &  & \\
 \omega_{1,4}            & \mapsto & \mu_{3,2}. &  &  & \\
\end{array}
\end{equation}
\medskip

Let $E_m\in H^*(BG_\Phi)$ be the equivariant Euler class of the normal bundle of $\eta_m$ at $\Phi$. The key point of our argument is the following  Proposition.

\begin{proposition} \label{prop:Euler}
The map $\iota_m^* \circ \phi_m:\PP_{\beta_1,m_1} \otimes \ldots \otimes \PP_{\beta_N,m_N} \to H^*(BG_\Phi)$ maps the element
$$
f_1(\omega_{i(1),1},\ldots,\omega_{i(1),m_1})
\otimes \ldots \otimes
f_N(\omega_{i(N),1},\ldots,\omega_{i(N),m_N})
$$
to
$$
f_1(\mu_{1,1},\ldots,\mu_{1,m_1})
\cdot \ldots \cdot
f_N(\mu_{N,1},\ldots,\mu_{N,m_N}) \cdot E_m.
$$
\end{proposition}

\begin{proof}
Applying the localization formula of Section \ref{sec:localization} for $f_1*\ldots*f_N$ we obtain a sum of
$$\prod_{i=1}^n  \frac{ \gamma(i)!}{{\gamma_1(i)!\gamma_2(i)! \ldots \gamma_N(i)!}}$$
terms. Recall that we interpreted this localization formula as the localization formula for the $\pi_*$ map in Lemma \ref{lem:star_int_impr}. Now we use Reineke's result \cite{reineke} again, claiming that $\Sigma$ is a resolution of the orbit closure $\overline{\eta}_m$. In particular there is only one torus fixed point over the point $\Phi$ in $\Sigma$. Therefore, when applying the $\iota_m^*$ map to this localization sum {\em all but one terms} will map to 0.

This one term corresponds to the choice of subsets $\cup_{v=1}^{m_u} Y_{i,u,v} \subset \{1,\ldots,\gamma(i)\}$. Hence the $\iota_m^*$-image of this term is
$$
f_1(\mu_{1,1},\ldots,\mu_{1,m_1})
\cdot \ldots \cdot
f_N(\mu_{N,1},\ldots,\mu_{N,m_N})
\cdot W,
$$
where $W$ is a ratio of products of linear factors of the type $\mu_{..}-\mu_{..}$, independent of the $f_u$'s. Since $W$ is independent of the $f_u$'s, we can find its value by choosing $f_u=1$ for all $u$. We obtain
$$\iota_m( 1 * \ldots * 1) = \iota_m^*( [\overline{\eta}_m] ) =E_m\qquad \qquad \text{and hence} \qquad\qquad W=E_m.$$
Here the first equality holds because of Theorem \ref{thm:tp}. The second equality follows from the topological observation that the class of a variety restricted to the smooth points of the variety itself is the Euler class of the normal bundle.

\medskip
\noindent{\sl Example:} Continuing the example above (with the choice of $i(1)=2, i(2)=1, i(3)=1$), we have
$$
f_1(\omega_{2,1},\omega_{2,2}) * f_2(\omega_{1,1},\omega_{1,2})*f_3(\omega_{1,1},\omega_{1,2})
=
$$
$$
f_1(\omega_{2,1},\omega_{2,2}) f_2(\omega_{1,1},\omega_{1,2}) f_3(\omega_{1,3},\omega_{1,4}) \cdot
\frac{
\prod_{i=1}^4 \prod_{j=1}^2 (\omega_{1,i}-\omega_{2,j})
\prod_{i=3}^4 \prod_{j=3}^4 (\omega_{1,i}-\omega_{2,j})}
{\prod_{i=3}^4 \prod_{j=1}^2 (\omega_{1,i}-\omega_{1,j})
\prod_{i=3}^4 \prod_{j=3}^4 (\omega_{2,i}-\omega_{2,j})}
$$
$$+\text{[35 similar terms]}.$$
The homomorphism $\iota_{2,2,2}^*$ (see (\ref{eqn:iota})) maps all the terms from term 2 to term 36 to 0, and maps term 1 given above to
$$
f_1(\mu_{1,1},\mu_{1,2}) f_2(\mu_{2,1},\mu_{2,2}) f_3(\mu_{3,1},\mu_{3,2}) \cdot
\frac{
\prod_{u=2}^3 \prod_{v=1}^2\prod_{v'=1}^2 (\mu_{u,v}-\mu_{1,v'})
\prod_{v=1}^2\mu_{v'=1}^2 (\mu_{3,v}-\mu_{2,v'})}
{
\prod_{v=3}^4\prod_{v'=1}^2 (\mu_{1,v}-\mu_{1,v'})
\prod_{v=3}^4\prod_{v'=1}^2 (\mu_{2,v}-\mu_{2,v'})}=
$$
$$
f_1(\mu_{1,1},\mu_{1,2}) f_2(\mu_{2,1},\mu_{2,2}) f_3(\mu_{3,1},\mu_{3,2}) \cdot
\underbrace{\prod_{v=1}^2 \prod_{v'=1}^2 (\mu_{3,v}-\mu_{1,v'})}_{E_{2,2,2}}.
$$
\medskip

This concludes the proof of the proposition. \end{proof}

Proposition \ref{prop:Euler} clearly implies the injectivity of the map $\phi_m$, and hence Lemma \ref{lem:inj1}.
\end{proof}

\subsection{Finishing the proof of Theorem \ref{thm:structure}}
We proved that the map (\ref{eqn:str2}) is injective. We claim that Theorem \ref{thm:dilog} can be interpreted as the fact that the two sides of (\ref{eqn:str2}) have the same dimensions (Poincar\'e series)---this is already implicit in \cite{ks_coha}. This implies Theorem \ref{thm:structure}. For completeness we now show the details of the claim that the two sides of (\ref{eqn:str2}) indeed have the same Poincar\'e series.

\medskip


Let $\HH_{m,k}^{A_1}$ be the degree $2k$ part of $\HH^{A_1}_m$, that is $H^{2k}(BGL_m)$. Consider the following twisted-shifted Poincar\'e series
$$h(z)=\sum_{m,k} \dim ( \HH_{m,k}^{A_1} )\cdot (-z)^m q^{m^2/2+k}.$$
Observe that this is exactly $\EE(z)$ from Section \ref{sec:dilog}.

Let $(\PP_{\beta,m})_k$ be the degree $2k$ part of $\PP_{\beta,m}$, that is, also $H^{2k}(BGL_m)$. Consider the Poincar\'e series
$$h_{\beta_u}(z)=\sum_{m,k} \dim ( \PP_{\beta_u,m} )_k \cdot (-z)^m q^{m^2/2+k}=h(z)=\EE(z).$$

Straightforward calculation shows that
\begin{equation} \label{eqn:E}
\EE(y_{\beta_1})\cdots \EE(y_{\beta_N}) =h_{\beta_1}(y_{\beta_1})\cdots h_{\beta_N}(y_{\beta_N})
\end{equation}
\begin{align}
& = \sum_{m_u} y_{\beta_1}^{m_1}\ldots y_{\beta_N}^{m_N} (-1)^{\sum m_u} q^{\sum m_u^2/2} \sum_K q^K
\underbrace{
\sum_{k_1+\ldots+k_N=K}\dim(\PP_{\beta_1,m_1})_{k_1} \cdots \dim(\PP_{\beta_N,m_N})_{k_N}.
}_
{\dim( \PP_{\beta_1,m_1}*\ldots* \PP_{\beta_N,m_N})_{K+\codim \eta_m}}  \nonumber \\ \label{eqn:w1}
& = \sum_{\gamma(1),\ldots,\gamma(n)}
y_{\alpha_1}^{\gamma(1)}\ldots y_{\alpha_n}^{\gamma(n)} (-1)^{\sum \gamma(i)} q^{\sum \gamma(i)^2/2} \sum_K q^K
\dim( \PP_{\beta_1,m_1}*\ldots* \PP_{\beta_N,m_N} )_K.
\end{align}
In the last equality we used Lemma \ref{lem:codim_in_AQ}. However, according to Theorem \ref{thm:dilog} expression (\ref{eqn:E}) can be written as
$$\EE(y_{\alpha_1})\cdots \EE(y_{\alpha_n}) =h_{\alpha_1}(y_{\alpha_1})\cdots h_{\alpha_n}(y_{\alpha_n})=$$
\begin{equation} \label{eqn:w2}
\sum_{\gamma(1),\ldots,\gamma(n)}
y_{\alpha_1}^{\gamma(1)}\ldots y_{\alpha_n}^{\gamma(n)} (-1)^{\sum \gamma(i)} q^{\sum \gamma(i)^2/2} \sum_K q^K
\dim( \HH_{\gamma} )_K.
\end{equation}
The comparison of (\ref{eqn:w1}) with (\ref{eqn:w2}) shows that the Poincar\'e series of the two sides of (\ref{eqn:str2}) are the same. Since we already proved that the map (\ref{eqn:str2}) is injective, we can conclude that it is an isomorphism. This finishes the proof of  Theorem \ref{thm:structure}.

\bibliography{coha}
\bibliographystyle{alpha}

\end{document}